\documentclass[11pt]{article}
\usepackage{latexsym,amssymb,amsmath,amsfonts,amsthm}
\usepackage{graphics,graphicx,mathrsfs,subfigure}
\usepackage{color,xspace}
\newcommand{\R}{{\mat R}}

\newcommand{\C}{{\mat C}}

\newcommand{\be}{\begin{eqnarray}}
\newcommand{\ben}{\begin{eqnarray*}}
\newcommand{\en}{\end{eqnarray}}
\newcommand{\enn}{\end{eqnarray*}}

\newcommand{\mat}{\mathbb}

\newtheorem{theorem}{Theorem}[section]
\newtheorem{lemma}[theorem]{Lemma}

\newtheorem{remark}[theorem]{Remark}

\definecolor{rot}{rgb}{1,0,0}
\definecolor{hw}{rgb}{0,0,1}

\begin{document}
\renewcommand{\theequation}{\arabic{section}.\arabic{equation}}
\title{\bf
 Increasing stability for inverse acoustic source
problems 
}
\author{ Suliang Si\thanks{School of Mathematics and Statistics, Shandong University of Technology,
Zibo, 255000, China ({\tt sisuliang@amss.ac.cn})}}
\date{}



\maketitle

\begin{abstract}
In this paper, we show the increasing stability of the inverse source problems for the acoustic wave
equation in the full space $\R^3$. The goal is to understand increasing
stability for wave equation in the time domain. If the time and spatial
variables of the source term can be separated with compact support, the increasing stability estimates of the $L^2$-norm of the acoustic source
function can be established. The stability estimates consist of
two parts: the Lipschitz type data discrepancy and the high time tail of the
source functions. As the time increases, the latter decreases
and thus becomes negligible. 
\end{abstract}

%


\section{Introduction}
\noindent 

Consider the time-dependent wave propagation with an acoustic source modeled by
\begin{equation}\label{eq1}
  \partial_t^2U(x,t) -\Delta U(x,t)=F(x,t),\qquad  x\in\mathbb{R}^3,\ t>0,
\end{equation}
where $U$ denotes the wave field and $F$ is the acoustic source. Together with the above governing equation, we impose the homogeneous initial conditions
\begin{equation}\label{eq2}
   U(x,0)=\partial_t U(x,0)=0,\qquad x\in\mathbb{R}^3.
\end{equation}
Specifically, we suppose that the dependence of the source term on time and 
space variables are separated, that is
\begin{equation}\nonumber
F(x,t)=f(x)g(t),
\end{equation}
where $f$ is compactly supported in the space region $B_R=\{x\in\R^3: \ |x|<R\}$ with $R>0$.

The well known representation for $u$ to the wave equation (\ref{eq1}) gives
\begin{equation}\label{U}
U(x,t)=\frac{1}{4\pi}\int_{\R^3}\frac{f(y)g(t-|x-y|)}{|y-x|}dy.
\end{equation}
In many applications, the temporal function $g(t)$ 
is usually given by the Ricker wavelet
\[g(t)=(1-2(\pi f_p(t-t_0))^2)e^{-(\pi f_p(t-t_0))^2}\]
with the center frequency $f_p$ and the delay time $t_0$.

In this
paper under the assumption that 
$g(t)=e^{-\gamma t}$ for constant $\gamma>0$,
we prove increasing stability for the inverse problem of recovering the source function $f(x)$ by the measurements, namely,
\begin{equation}\nonumber
u(x,t), \quad (x,t)\in\partial B_R\times(0,T),
\end{equation}
where  $T>1$.

Throughout the paper, $a\lesssim b$ stands for $a\leq Cb$, where $C>0$ is a constant
independent of $\gamma$, $R$. Now we introduce the main stability result of this paper.
\begin{theorem}\label{1}
Let $\|f\|_{L^2{(\R^3)}}\leq M$ and $g(t)=e^{-\gamma t}$ for constant $M, \gamma>0$. Assume that $U$ is the solution of the scattering problem (\ref{eq1})–(\ref{eq2}). Then there exist constant $\alpha>0$ and $C>0$ depending on $R$, $\gamma$, $\alpha$ such that 
\begin{equation}\label{f}
 \|f\|^2_{L^2{(\R^3)}}\leq C\Big(\epsilon^2+\frac{ M^2}{(T^{\frac{2}{3}}|\ln \epsilon|^{\frac{1}{4}})^{\alpha}}\Big),
\end{equation}
where $T>1$ and 
\[\epsilon=\Big(\int_0^T\int_{\partial B_R}\big(|\partial^2_tU(x,t)|^2+|\partial_\nu\partial_tU(x,t)|^2\big)ds(x)dt\Big)^{\frac{1}{2}}.\]
\end{theorem}
\begin{remark}
There are two parts in the stability estimates $(\ref{f})$: the first part is the data
discrepancy and the second part comes from the high time tail of the function. It is clear to
see that the stability increases as the time $T$ increases, i.e., the source $f(x)$ is more stable as more
time domain data are used. Our stability result in this work is consistent with the one in \cite{CIL2016} for both the two-
and three-dimensional inverse scattering problems.
\end{remark}

\subsection{Known results and Motivations}
Inverse sourse problems have many significant applications in scientific and engineering areas.
For instance, detection of submarines and non-destructive measurement of industrial objects
can be regarded as recovery of acoustic sources from boundary measurements of the pressure.
Other application include biomedical imaging optical tomography \cite{[3],Isa1990} and geophysics. For
a mathematical overview of various inverse source problems, one can see that uniqueness and
stability are discussed in \cite{Isa1990}. For inverse source problems in time domain, it is solved as
hyperbolic systems by using Carleman estimate \cite{K2002} and unique continuation; we refer to \cite{CY2006,JLY2017,Y1995,Y1999} for an incomplete list. For time-harmonic inverse source problems, it is well-known that
there is no uniqueness for the inverse source problem with a single frequency due to the existence
of non-radiating sources \cite{BC,HKP2005}. Therefore, the use of multiple frequencies data is an effective
way to overcome non-uniqueness and has received a lot of attention in recent years. \cite{[12]} showed
the uniqueness and numerical results for Helmholtz equation with multi-frequency data. In \cite{[5]},
Bao et al. initiated the mathematical study on the stability of the inverse source problem for the
Helmholtz equation by using multi-frequency data. Then in \cite{CIL2016}, a different method involving a
temporal Fourier transform, sharp bounds of the analytic continuation to higher wave numbers
was used to derive increasing stability bounds for the three dimensional Helmholtz equation.
Also, \cite{CIL2016} firstly combined the Helmholtz equation and associated hyperbolic equations to get
the stability results. Later in \cite{[20]} and \cite{BLZ2020}, increasing stability was extended to Helmholtz
equation and Maxwell’s equation in three dimension. We refer to \cite{ACTV2013, ZL2019, AHS2024, S2024} for the study of the
inverse source problems by using multiple frequency information.

Although a lot of work has been done on increasing stability of the inverse source problem for the time-harmonic waves. However, acoustic wave equation can also describe the propagation of waves. Does the inverse source problem of the acoustic equation also have some increasing stability?  This work
initializes the mathematical study and provides the first stability results of the inverse source problems
for acoustic waves. The next section  is devoted
to the proof of the result.

\section{Proof of Theorem \ref{1}}\label{sec2}

The following lemma gives bounds for $f$ within the framework of the hyperbolic initial
value problem (\ref{eq1}). 
\begin{theorem}\label{BK}
Let $u$ be the solution of the scattering problem (\ref{eq1})–(\ref{eq2}). Then 
\begin{equation}\nonumber
\|f\|^2_{L^2(\R^3)}\leq C\Big(\|\partial^2_tu\|^2_{L^2(\partial B_R\times(0,\infty))}+\|\partial_\nu\partial_tu\|^2_{L^2(\partial B_R\times(0,\infty))}\Big),
\end{equation}
where $C>0$ depends on $R$ and $\gamma$.
\end{theorem}

\begin{proof}
Let $\xi\in\R^3$ with $|\xi|=w\in(0,\infty)$.
Our argument relies essentially on the temporal
Fourier transform of the function $U$, deﬁned by
\begin{equation}\nonumber
u(x,w)=\int_{0}^{\infty}U(x,t)e^{-iwt}dt.
\end{equation}
Multiplying $e^{-i(\xi\cdot x+wt)}$ on both sides of (\ref{1})
and integrating over $B_R\times(0,\infty)$, we obtain
\begin{equation}\nonumber
\begin{split}
\int_{B_R}f(x)e^{-i\xi\cdot x}dx\int_{0}^{\infty}g(t)e^{-iwt}dt
&=-\int_0^\infty\int_{\partial B_R}e^{-i\xi\cdot x}(\partial_\nu U(x,t)e^{-iwt}+i(\nu\cdot\xi)U(x,w)e^{-iwt})ds(x)dt\\
&=-\int_{\partial B_R}e^{-i\xi\cdot x}(\partial_\nu u(x,w)+i(\nu\cdot\xi)u(x,k))ds(x).
\end{split}
\end{equation}
Here we used that the temporal function $g(t)=e^{-\gamma t}$ is time decay.

Since $|\int_{0}^{\infty}g(t)e^{-iwt}dt|\geq c_0>0$ and Supp$f\subset B_R$, we have
\begin{equation}\nonumber
|\int_{\R^3}f(x)e^{-i\xi\cdot x}dx|^2\leq C\int_{\partial B_R}(\big|\partial_\nu u(x,w)|^2+w^2|u(x,w)\big|^2)ds(x).
\end{equation}
Hence we obtain by using the spherical coordinates that
\begin{equation}
\begin{split}
\int_{\R^3}|\int_{\R^3}f(x)e^{-i\xi\cdot x}dx|^2d\xi&\leq C \int_{\R^3}\int_{\partial B_R}\big|\partial_\nu u(x,w)|^2+w^2|u(x,w))\big|^2ds(x)\\
&\leq C\int_0^{2\pi}d\theta\int_0^{\pi}\sin\varphi d\varphi\int_0^{\infty}w^2\big(\int_{\partial B_R}|\partial_\nu u(x,w)|^2+w^2|u(x,w))|^2ds(x)\big)dw\\
&\leq C\int_0^{\infty}(\int_{\partial B_R}|w\partial_\nu u(x,w)|^2+|w^2u(x,w))|^2)ds(x)dw.
\end{split}
\end{equation}
Using the Parseval’s identity yields
\begin{equation}\nonumber
\int_0^\infty\int_{\partial B_R}(|\partial^2_tu|^2ds(x)dt+|\partial_\nu\partial_tu|^2)ds(x)dt=\int_0^{\infty}\int_{\partial B_R}(|w^2u(x,w)|^2+|w\partial_\nu u(x,w)|^2)ds(x)dw.
\end{equation}
It follows from the Parseval’s identity again  that
\begin{equation}
\|f\|^2_{L^2(\R^3)}\leq C\Big(\|\partial^2_tu\|^2_{L^2(\partial B_R\times(0,\infty))}+\|\partial_\nu\partial_tu\|^2_{L^2(\partial B_R\times(0,\infty))}\Big),
\end{equation}
which completes the proof.

\end{proof}
Let
\begin{equation}\nonumber
I_1(k)=\int_0^{k}\int_{\partial B_R}|\partial^2_tu(x,t)|^2ds(x)dt
\end{equation}
and
\begin{equation}\nonumber
I_2(k)=\int_0^k\int_{\partial B_R}|\partial_\nu\partial_tu(x,t)|^2ds(x)dt.
\end{equation}
Denote
\[S=\{z=x+iy\in\C: -\frac{\pi}{4}<z<\frac{\pi}{4}\}.\]
Since $g(t)$ is entire analytic function of $t$, the integrals $I_1(k)$ and $I_2(k)$ with respect
to $t$ can be taken over any path joining points $0$ and $k$ of the complex plane. Thus $I_1(k)$ and $I_2(k)$ are entire analytic functions of $k=k_1+ik_2\in\C$ ($k_1,k_2\in\R$) and the following elementary estimates hold.

\begin{lemma}
Let $f\in L^2(\R^3)$, $supp f\subset B_R$. For any $k=k_1+ik_2\in S$, we have
\begin{equation}\label{eIk}
\begin{split}
|I_1(k)|\leq C\|f\|_{ L^2(\R^3)}^2|k|e^{2\gamma|k_1|}
\end{split}
\end{equation}
and 
\begin{equation}\label{eIk2}
\begin{split}
|I_1(k)|\leq C\|f\|_{ L^2(\R^3)}^2|k|e^{2\gamma|k_1|},
\end{split}
\end{equation}
where $C>0$ depends on $R$ and $\gamma$.
\end{lemma}
\begin{proof}
Since
\begin{equation}\nonumber
U(x,t)=\frac{1}{4\pi}\int_{\R^3}\frac{f(y)g(t-|x-y|)}{|y-x|}dy,
\end{equation}
we get
\begin{equation}\nonumber
\partial^2_tU(x,t)=\frac{1}{4\pi}\int_{\R^3}\frac{f(y)g''(t-|x-y|)}{|y-x|}dy.
\end{equation}
Let $t=k\hat{t}$, $\hat{t}\in(0,1)$. It follows from the change of
variables that
\begin{equation}\nonumber
\begin{split}
I_1(k)=&\int_0^{k}\int_{\partial B_R}|\partial^2_tu(x,t)|^2ds(x)dt\\
&=\int_0^{k}\int_{\partial B_R}\big|\frac{1}{4\pi}\int_{B_R}\frac{f(y)g''(t-|x-y|)}{|y-x|}dy\big|^2ds(x)dt\\
&=\int_0^{1}\int_{\partial B_R}k\big|\frac{1}{4\pi}\int_{B_R}\frac{f(y)g''(k\hat{t}-|x-y|)}{|y-x|}dy\big|^2ds(x)d\hat{t}.
\end{split}
\end{equation}
Using the Cauchy–Schwarz inequality yields
\begin{equation}\nonumber
\begin{split}
|I_1(k)|&\leq C\int_0^{1}|k|\int_{\partial B_R}|f(y)|^2dy\int_{\partial B_R} \big|\frac{1}{4\pi}\int_{B_R}\frac{g''(k\hat{t}-|x-y|)}{|y-x|}dy\big|^2ds(x)d\hat{t}\\
&\leq C\int_0^{1}|k|\int_{\partial B_R}|f(y)|^2dy\int_{\partial B_R} |\frac{1}{4\pi}\int_{B_R}\frac{e^{-2\gamma(k\hat{t}-|x-y|)}}{|y-x|^2}dy|ds(x)d\hat{t}\\
&\leq C|k|e^{\gamma k_1}\|f\|^2_{L^2{(B_R)}}.
\end{split}
\end{equation}
Noting that 
\begin{equation}\nonumber
\partial_\nu\partial_tu(x,t)=\frac{1}{4\pi}\int_{\R^3}\frac{f(y)\big(g''(t-|x-y|)+g'(t-|x-y|)\big)\frac{y-x}{|y-x|}\cdot\nu}{|y-x|^2}dy,
\end{equation}
we have from the Cauchy–Schwarz inequality and the polar coordinates $\rho=|y-x|$ (originated at $x$) with respect to $y$ yield that
\begin{equation}\nonumber
|\partial_\nu\partial_tu(x,t)|^2\leq e^{-2\gamma t}|\int_0^{2R}f(y)e^{\gamma\rho}d\rho|^2\leq C\|f\|^2_{L^2{(B_R)}}e^{-2\gamma t}.
\end{equation}
Similarly we obtain 
\begin{equation}\label{IM}
\begin{split}
|I_2(k)|=|k\int_0^1\int_{\Gamma_0}|\partial_\nu\partial_tu(x,k\hat{t})|^2ds(x)d\hat{t}|
\leq C\|f\|^2_{L^2{(B_R)}}|k|e^{2\gamma|k_1|},
\end{split}
\end{equation}
which completes the proof.
\end{proof}

\begin{lemma}
For any $\alpha>0$, we have
\begin{equation}\nonumber
\begin{split}
\int_k^{\infty}\int_{\partial B_R}|\partial^2_tu(x,t)|^2ds(x)dt+\int_k^{\infty}\int_{\partial B_R}|\partial_\nu\partial_tu(x,t)|^2ds(x)dt\leq C\frac{\|f\|^2_{L^2(\R^3)}}{k^\alpha},
\end{split}
\end{equation}
where $C>0$ depends on $R$, $\gamma$ and $\alpha$.
\end{lemma}

\begin{proof}
Using $e^{-2\gamma t}\leq C\frac{1}{t^{\alpha+1}}$, $t>0$, where $C>0$ depends on $\gamma$ and $\alpha$,
we have
\begin{equation}\nonumber
\begin{split}
\int_k^{\infty}\int_{\partial B_R}|\partial^2_tu(x,t)|^2ds(x)dt
&=\int_k^{\infty}\int_{\partial B_R}|\frac{1}{4\pi}\int_{B_R}\frac{f(y)g''(t-|x-y|)}{|y-x|}dy|^2ds(x)dt\\
&\leq C\frac{\|f\|^2_{L^2(\R^3)}}{s^\alpha}.
\end{split}
\end{equation}
Similarly one may show that
$\int_k^{\infty}\int_{\partial B_R}|\partial_\nu\partial_tu(x,t)|^2ds(x)dt$ have similar exponential decay estimates.

\end{proof}

Let us recall the following result, which is proved in \cite{CIL2016}.
\begin{lemma}\label{2}
  Let $J(z)$ be an analytic function in $S=\{z=x+iy\in \mathbb{C}:-\frac{\pi}{4}<\arg z<\frac{\pi}{4}\}$ and continuous in $\overline{S}$ satisfying
  \begin{equation}
  \nonumber
  \begin{cases}
    |J(z)| \leq \epsilon, \  & z\in (0,\ L], \\
    |J(z)| \leq V, \  & z\in S,\\
    |J(0)|  =0.
  \end{cases}
  \end{equation}
Then there exists a function $\mu(z)$ satisfying
  \begin{equation}
  \nonumber
  \begin{cases}
   \mu (z)  \geq\frac{1}{2},\ \ &z\in (L,\ 2^{\frac{1}{4}}L), \\
   \mu (z)  \geq\frac{1}{\pi}((\frac{z}{L})^4-1)^{-\frac{1}{2}},\ \ & z\in (2^{\frac{1}{4}}L, \ +\infty)
  \end{cases}
  \end{equation}
   such that
\begin{equation}
\nonumber
 |J(z)|\leq V\epsilon^{\mu(z)} \quad \mbox{for all} \quad z\in (L, \ +\infty).
\end{equation}
\end{lemma}
Let $I(k)=I_1(k)+I_2(k)$. Using Lemma $\ref{2}$, we show the relation between $I(k)$ for $s\in(T,\infty)$ with $I(T)$.
\begin{lemma}
Let $||f||_{L^2(\R^3)}\leq M$. Then there exists a function $\mu(s)$ satisfying 
  \begin{equation}
  \nonumber
  \begin{cases}
   \mu (k)  \geq\frac{1}{2},\ \ &k\in (T,\ 2^{\frac{1}{4}}T), \\
   \mu (k)  \geq\frac{1}{\pi}((\frac{s}{T})^4-1)^{-\frac{1}{2}},\ \ & k\in (2^{\frac{1}{4}}T, \ +\infty)
  \end{cases}
  \end{equation}
such that 
\begin{equation}\label{CM}
|I(k)|\leq CM^2e^{(2\gamma+1)k}\epsilon^{2\mu(k)} \quad \mbox{for all}
\ T<k<+\infty,
\end{equation}
where $C>0$ depends on $R$ and $\gamma$.
\end{lemma}
\begin{proof}
Let the sector  $S\subset\C$ be given in Lemma \ref{2}. Observe that $|k_2|\leq k_1$ when $k\in S$. It follows from (\ref{IM}) that 
\[|I(k)e^{-(2\gamma+1)k}|\leq CM^2,\]
where $C>0$ depends on $R$ and $\gamma$.
Since $\epsilon^2=\int_0^T\int_{\Gamma_0}|\partial_\nu\partial_tu(x,t)|^2ds(x)dt$, we have
\begin{equation}\label{I1k}
|I(k)|\leq\epsilon^2, \quad k\in[0, T].
\end{equation}
Then applying Lemma \ref{2} with $L=k$ to be function $J(s):=I(k)e^{-(2\gamma+1)k}$,
we conclude that there exists a function $\mu(k)$ satisfying 
\begin{equation}
\nonumber
\begin{cases}
\mu (k)  \geq\frac{1}{2},\ \ &k\in (T,\ 2^{\frac{1}{4}}T), \\
\mu (k)  \geq\frac{1}{\pi}((\frac{s}{T})^4-1)^{-\frac{1}{2}},\ \ & k\in (2^{\frac{1}{4}}T, \ \infty)
\end{cases}
\end{equation}
such that
\begin{equation}
|I(k)e^{-(2\gamma+1)k}|\leq CM^2\epsilon^{2\mu(k)},
\end{equation}
where $T<k<+\infty$ and $C$ depends on $R$ and $\gamma$. Thus we complete the proof.
\end{proof}

Now we show the proof of Theorem \ref{1}.

If $\epsilon\geq e^{-1}$, then the estimate is 
obvious.
If $\epsilon<e^{-1}$, we discuss (\ref{f}) in two cases.

Case (i):  $2^{\frac{1}{4}}((2\gamma+3)\pi)^{\frac{1}{3}} T^{\frac{1}{3}}< |\ln \epsilon|^{\frac{1}{4}}$. 
Choose $s_0=\frac{1}{((2\gamma+3)\pi)^{\frac{1}{3}}}T^{\frac{2}{3}}|\ln \epsilon|^{\frac{1}{4}}$. It is easy to get $s_0>2^{\frac{1}{4}}T$, then 
\[-\mu(s_0)\leq-\frac{1}{\pi}((\frac{s_0}{T})^4-1)^{-\frac{1}{2}}\leq-\frac{1}{\pi}(\frac{T}{s_0})^2.\]
A direct application of  estimate (\ref{CM}) shows that 
 \begin{eqnarray*}
 |I(s_0)|&\leq& CM^2\epsilon^{2\mu(s_0)}e^{(2\gamma+3)s_0}\\
        &\leq& CM^2 e^{(2\gamma+3)s_0-2\mu(s_0)|\ln\epsilon|}\\
           &\leq& CM^2 e^{(2\gamma+3)s_0-\frac{2|\ln \epsilon|}{\pi}(\frac{T}{s_0})^2 }\\
           &=& CM^2 e^{-2(\frac{(2\gamma+3)^2}{\pi})^{\frac{1}{3}}T^{\frac{2}{3}}|\ln \epsilon|^{\frac{1}{2}}(1- \frac{1}{2}|\ln \epsilon|^{-\frac{1}{4}})},
 \end{eqnarray*}
where $C>0$ depends on $R$ and $\gamma$. 
Noting that $1-\frac{1}{2}|\ln \epsilon|^{-\frac{1}{4}}>\frac{1}{2}$ and $(\frac{(2\gamma+3)^2}{\pi})^{\frac{1}{3}}>1$, we have 
\[|I(s_0)|\leq CM^2 e^{-T^{\frac{2}{3}}|\ln \epsilon|^{\frac{1}{2}}}.\]
Using the inequality $e^{-t}\leq C\frac{1}{t^{3\alpha}}$ for $t>0$, we get
\begin{equation}\label{ess}
  |I(s_0)|\leq C\frac{M^2 }{(T^2|\ln \epsilon|^{\frac{3}{2}})^{\alpha}}.
\end{equation}
Hence there exists $C>0$ depending on $R$, $\alpha$ and $\gamma$ such that
\begin{equation}\label{ess1}
  \begin{split}
    \|f\|^2_{L^2{(\R^3)}}&= I(s_0)+\int_{s_0}^{\infty}\int_{\partial B_R}\big(|\partial^2_tu(x,t)|^2+|\partial_\nu\partial_tu(x,t)|^2\big)ds(x)dt\\
    &\leq C\frac{M^2 }{(T^2|\ln \epsilon|^{\frac{3}{2}})^{\alpha}}+C\frac{M^2 }{(T^{\frac{2}{3}}|\ln \epsilon|^{\frac{1}{4}})^{\alpha}}\\
    &\leq C\frac{M^2 }{(T^{\frac{2}{3}}|\ln \epsilon|^{\frac{1}{4}})^{\alpha}}.
  \end{split}
\end{equation}
Since $T^2|\ln \epsilon|^{\frac{3}{2}} \geq T^{\frac{2}{3}}|\ln \epsilon|^{\frac{1}{4}}$ when $T>1$ and $|\ln\epsilon|\geq1$.

Case (ii):  $|\ln \epsilon|^{\frac{1}{4}}\leq 2^{\frac{1}{4}}((2\gamma+3)\pi)^{\frac{1}{3}} T^{\frac{1}{3}}$. In this case we choose $s_0=T$, then $s_0\geq 2^{-\frac{1}{4}}((2\gamma+3)\pi)^{-\frac{1}{3}}T^{\frac{2}{3}}|\ln \epsilon|^{\frac{1}{4}}$. 
Using  estimate (\ref{I1k}), we obtain
\begin{equation}\label{ess2}
  \begin{split}
    \|f\|^2_{L^2{(\R^3)}}&= I(s_0)+\int_{s_0}^{\infty}\int_{\partial B_R}\big(|\partial^2_tu(x,t)|^2+|\partial_\nu\partial_tu(x,t)|^2\big)ds(x)dt\\
&\leq C\big(\epsilon^2+\frac{ M^2}{(T^{\frac{2}{3}}|\ln \epsilon|^{\frac{1}{4}})^{\alpha}}\big),
\end{split}
\end{equation}
where $C>0$ depends on $\alpha$, $\gamma$ and $\alpha$.
Combining (\ref{ess1}) and (\ref{ess2}), we finally get
\begin{equation}\nonumber
 \|f\|^2_{L^2{(\R^3)}}\leq C\big(\epsilon^2+\frac{ M^2}{(T^{\frac{2}{3}}|\ln \epsilon|^{\frac{1}{4}})^{\alpha}}\big),
\end{equation}
where $C>0$ depends on $R$, $\gamma$ and $\alpha$.

\section*{Acknowledgment}
The work of Suliang Si is supported by  the Shandong Provincial Natural Science Foundation (No. ZR2022QA111).

\end{document}